\documentclass[12pt]{amsart}
\usepackage{latexsym}
\usepackage{amsthm}
\usepackage{amssymb}
\usepackage[utf8]{inputenc}
\usepackage[T1]{fontenc}
\usepackage[all]{xy}
\usepackage{amsfonts}
\usepackage{amsmath}
\usepackage{graphicx}
\usepackage{pstricks}
\usepackage{url}
\usepackage{hyperref}
\usepackage[a4paper]{geometry}
\author[A. Nibirantiza]{Aboubacar Nibirantiza} 
\address[Aboubacar Nibirantiza]{University of Burundi, Institute of Applied Pedagogy, Department of mathematics, B.P 2523, Bujumbura-Burundi}
\email{aboubacar.nibirantiza[at]ub.edu.bi}
\date{\today} 
\title[The fine $\spo(2|n)$-equivariant quantizations on the super circles $S^{1|n}$]{The fine $\spo(2|n)$-equivariant quantizations on the super circles $S^{1|n}$}

\newtheorem{theorem}{Theorem}[section]

\newtheorem{df}[theorem]{Definition}

\newtheorem{rmk}[theorem]{Remark}
\newtheorem{lem}[theorem]{Lemma}
\newtheorem{prop}[theorem]{Proposition}

\def\d{\delta}

\newcommand{\R}{\mathbb{R}}

\newcommand{\N}{\mathbb{N}}

\newcommand{\D}{\mathcal{D}}

\newcommand{\Id}{\mathrm{Id}}
\newcommand{\spo}{\mathfrak{spo}}
\newcommand{\pgl}{\mathfrak{pgl}}

\begin{document}
\maketitle
\begin{abstract}
In this paper, we generalize the known results on the super circles $S^{1|1}$ and $S^{1|2}$. We construct the fine equivariant quantization on the super circle $S^{1|n}$ for $n\geqslant 3$. The equivariant Lie superalgebra is $\spo(2|n)$ which is constituted of the contact projective vector fields on $S^{1|n}$. In order to construct the fine equivariant quantization on $S^{1|n}$, we use the model developed, in the purely even case, by  Charles H. Conley and Valentin Ovsienko in \textit{[Linear Differential Operators on Contact manifolds, www.arxiv:math-Ph/1205.6562v1,24p, 2012]}. We also use the technical of Casimir operators to prove the uniqueness of the fine quantization on $S^{1|n}$.
The technical of Casimir operators used here is the same as the one used by P. Mathonet and F. Radoux  in [\textit{Lett. Math. Phys. 98 (2011),311-331}] to prove the existence of a $\pgl(p+1|q)$-equivariant quantization on $\R^{p|q}$.
\end{abstract}
\section{Introduction}
The concept of equivariant quantization over $\mathbb{R}^{n}$ was introduced by P. Lecomte and V. Ovsienko in \cite{LO}. An equivariant quantization is a linear bijection between a space of differential operators and its corresponding space of symbols that commutes with the action of a Lie subalgebra of vector fields over $\mathbb{R}^{n}$ and preserves the principal symbol.

In their seminal work \cite{LO}, P. Lecomte and V. Ovsienko considered spaces of differential operators acting between densities and the Lie algebra of projective vector fields over $\mathbb{R}^{n}$, $\mathfrak{sl}(n+1)$. In this situation, they showed the existence and uniqueness of an equivariant quantization.

The results of \cite{LO} were generalized in many directions and recently several papers dealt with the problem of equivariant quantizations in the
context of supergeometry. For example, the thesis \cite{Mic09} dealt with \emph{conformally equivariant quantizations} over supercotangent bundles, the papers \cite{MatRad11} and \cite{LMR} exposed and solved respectively the problems of the $\mathfrak{pgl}(p+1|q)$-equivariant quantization over $\mathbb{R}^{p|q}$ and of the $\mathfrak{osp}(p+1,q+1|2r)$-equivariant quantization over $\mathbb{R}^{p+q|2r}$, whereas in \cite{LR}, the authors define the problem of the natural and projectively invariant quantization on arbitrary supermanifolds and show the existence of such a map.\\
Another type of equivariant quantization was studied over the super circles $S^{1|1}$ and $S^{1|2}$ endowed with the canonical contact structures in \cite{GarMelOvs07,Mel09,MelNibRad13}. The equivariant Lie superalgebras are $\spo(2|1)$ on $S^{1|1}$ and $\spo(2|2)$ on $S^{1|2}$. These Lie superalgebras provide to the existence of the two geometric structures  on the super circle, i.e, the projective structure and the contact structure. In \cite{GarMelOvs07,Mel09,MelNibRad13} the explicit formulas of the equivariant quantization on the super circles  $S^{1|1}$ and $S^{1|2}$ are given.\\

In this paper, we generalize on the super circles $S^{1|n}$ for an arbitrary $n\geqslant 3$ the results obtained on  $S^{1|1}$ and $S^{1|2}$.
 The equivariant quantizations that we construct here intertwine the actions of the Lie superalgebra of contact projective vector fields $\spo(2|n)$ on $S^{1|n}$. This Lie superalgebra is the intersection of the Lie superalgebra of contact vector fields $\mathcal{K}(n)$ on $S^{1|n}$ and the Lie superalgebra of projective vector fields $\pgl{(2|n)}$.  In this paper, the spaces of differential operators are endowed with filtrations which are defined thanks to the contact structures and which are finer than the classical ones. The spaces of symbols are then the graded spaces corresponding to these finer filtrations. \\
 
 We use here two formalisms: on the one hand, we use the technical developed in the purely even case by Charles H. Conley and Valentin Ovsienko in \cite{CoOv12}: to prove the existence of equivariant quantization on the super circles $S^{1|n}$, we consider the symbols represented by the polynomials and the Lie derivative represented by the differential operators. On the other hand, we use the method of Casimir operators to prove the uniqueness of the equivariant quantization in the same way as the one linked to the Casimir operators used by P. Mathonet and F. Radoux in \cite{MatRad11} to built the $\pgl(p+1|q)$-equivariant quantization on $\R^{p|q}$. \\
 
 The paper is organized as follows. In Section 2, we recall the definitions of the objects that occur in the problem of quantization such as densities, differential operators and symbols and some tools exposed in \cite{CoOv12} to define the fine symbols. In Section 3, we expose the tools that we are going to use to build the fine quantization and we show the computations for the existence of the fine $\spo(2|n)$-equivariant quantization. In the section 4, we compute the Casimir operators on the symbols and on the differential operators in order to prove the uniqueness of the fine $\spo(2|n)$-equivariant quantization on the super circles $S^{1|n}$ for an arbitrary integer $n\geqslant 3$. \\
 We mention here that our construction of the fine equivariant quantization is not applicable in the case of $S^{1|2}$ because the $\mathfrak{sl}(2|2)$-equivariant quantization doesn't induce a $\spo(2|2)$-equivariant quantization.       

\section{Notation and problem setting}

In this section, we recall some tools pertaining to the problem of equivariant quantization such as weighted densities, differential operators, symbols, contact projective vector fields on $S^{1|n}$. Our description is the generalization of the one described in \cite{MelNibRad13}.
\subsection{Superfunctions on $S^{1|n}$}
We define the geometry of the superspace $S^{1|n}$, where $n\in\N^*$, by describing its associative supercommutative superalgebra of superfunctions on $S^{1|n}$ which we denote by $C^\infty(S^{1|n})$ and which is constituted by the elements 
 \begin{align*}
f(x,\theta)&=\sum_{0\leqslant|I|\leqslant q}{f_I(x)\theta_I}\\
&=f_0(x)+f_1(x)\theta_1+...+f_n(x)\theta_n+f_{12}(x)\theta_1\theta_2+...
+f_{1...n}(x)\theta_1...\theta_n
\end{align*}
where $|I|$ is the length of $I$, $x$ is the coordinate corresponding to one of the affine coordinates system on $\R P^1$, $\theta=(\theta_i),\quad i=1,\cdots, n$ is odd Grassmann coordinates, i.e. $\theta_i^2=0,\quad \theta_i\theta_j=-\theta_j\theta_i$ and where $f_I(x)\in C^\infty(S^1)$ are functions on $S^1$. We define the parity function $\tilde{.}$ by setting $\tilde{x}=0$ and $\tilde{\theta}=1$.

\subsection{Vector fields on $S^{1|n}$}
A vector field on $S^{1|n}$ is a superderivation of the associative supercommutative superalgebra $C^\infty(S^{1|n})$. In coordinates, it can be expressed as 
\[
X=f\partial_{x}+\sum_{i=1}^ng^i\partial_{\theta_i},
\]
where $f$ and $g^i$ are the elements of  $C^\infty(S^{1|n})$, $\partial_{x}=\frac{\partial}{\partial x}$ and $\partial_{\theta_i}=\frac{\partial}{\partial \theta_i}$ for all $i=1,2,\cdots,n$. \\

The parity function $\tilde{.}$ on vector field $X$ is defined as 
\[
\widetilde{\partial_{x_i}}=0\quad \mbox{and}\quad \widetilde{\partial_{\theta_i}}=1.
\]
The superspace of all vector fields on $S^{1|n}$ is a Lie superalgebra which we shall denote by $\mathrm{Vect}(S^{1|n})$.

\subsection{The Lie superalgebra of contact vector fields on $S^{1|n}$}

The standard contact structure on $S^{1|n}$ is defined by the linear distribution $$\langle\overline{D}_1,\overline{D}_2,\ldots,\overline{D}_n \rangle,$$ i.e  generated by the following odd vector fields
\[
\overline{D}_1=\partial_{\theta_1}-\theta_1\partial_x,\quad \overline{D}_2=\partial_{\theta_2}-\theta_2\partial_x,\quad\ldots,\quad\overline{D}_n=\partial_{\theta_n}-\theta_n\partial_x.
\]
It can also be seen that the contact structure on $S^{1|n}$ is spanned by the kernel of the $1$-form on $S^{1|n}$ defined by 
$$\alpha=dx+\sum_{i=1}^n\theta_id\theta_i.$$
\begin{df}\label{ContactVector}
A vector field on $S^{1|n}$ is called a contact vector field if it preserves the contact distribution, that is, satisfies the condition:
\[
[X,\overline{D}_1]=\sum_i\psi_{X_i}\overline{D}_i,\quad [X,\overline{D}_2]=\sum_i\phi_{X_i}\overline{D}_i,\quad\ldots,\quad [X,\overline{D}_n]=\sum_i\sigma_{X_i}\overline{D}_i,
\] where $\psi_{X_i},\phi_{X_i},\ldots \sigma_{X_i}$ are functions depending on $X$. The superspace of contact vector fields is a Lie superalgebra which we shall denote by $\mathcal{K}(n)$\footnote{ The notation $\mathcal{K}(1|n)$ can be also used in place of $\mathcal{K}(n)$.}. 
\end{df}
It is well-known (see in \cite{Nib15}) that every contact vector field on $S^{1|n}$ can be expressed, for some functions $f\in C^\infty(S^{1|n})$, by
\begin{equation}\label{Hamilton}
X_f= f\partial_x-(-1)^{\tilde{f}}\frac{1}{2}\sum_{i=1}^n\overline{D}_i(f)\overline{D}_i.
\end{equation}
The superfunction $f$ is called a contact Hamiltonian of the vector field $X_f$.
The following result is important.
\begin{prop}
For a fixed  $m\in\mathbb{N}^*$, the condition given in definition \ref{ContactVector} becomes
$$
[X_f,\overline{D}_m]=\frac{1}{2}\sum_{i=1}^n\overline{D}_m\overline{D}_i(f)\overline{D}_i. 
$$
\end{prop}
\begin{proof}
It is sufficient to compute the formula given in definition \ref{ContactVector} while replacing $X$ by the formula \eqref{Hamilton} of $X_f$.
\end{proof}
\begin{rmk}
It is obvious that $[X_f,\overline{D}_m]=\psi_{1}\overline{D}_1+\psi_2\overline{D}_2+\ldots+\psi_n\overline{D}_n$, where $\psi_i\in C^\infty(S^{1|n})$ for all $i=1,\ldots, n$ depending on $X_f$.
\end{rmk}
When we define the following Lagrange bracket on $C^\infty(S^{1|n})$
\begin{equation}\label{Lagrange}
\{f,g\}=fg'-f'g-(-1)^{\tilde{f}}\frac{1}{2}\sum_{i=1}^n\overline{D}_i(f)\overline{D}_i(g),
\end{equation} where $f'=\partial_xf$, we obtain a structure of Lie superalgebra on $C^\infty(S^{1|n})$.
We can thus identify the Lie superalgebra $C^\infty(S^{1|n})$ with the Lie superalgebra $\mathcal{K}(n)$
\subsection{The Lie superalgebra $\spo(2|n)$}
As done in \cite{Nib15,Nib14}, one can see that the Lie superalgebra $\spo(2|n)$ is the intersection of the Lie superalgebra $\mathcal{K}(n)$ and the Lie superalgebra of projective vector fields $\pgl{(2|n)}$ exposed in \cite{MatRad11}. The Lie superalgebra $\spo(2|n)$ is then called the Lie superalgebra of the contact projective vector fields. 
Thus $\spo(2|n)$ is a $(n+2|2n)$-dimensional Lie superalgebra spanned by the following  contact projective vector fields
\[
2X_x,X_{x^2}, -X_1, 2X_{\theta_{j-1}\theta_j},2X_{\theta_i},-2X_{x\theta_i},\quad i=1,\ldots,n\quad\mbox{and}\; j=2,\ldots,n
\]
Explicitly, we obtain (see \cite{{MelNibRad13}, {Nib15},{Nib14},{MatRad11}}) those vector fields when we realize the embedding of a Lie superalgebra of matrices belonging to $\mathfrak{gl}(2|n)$ into $\mathrm{Vect}(S^{1|n})$.
The Lie superalgebra spanned by the contact projective vector fields associated with the contact Hamiltonians $$\{1,x,\theta_i,\theta_{j-1}\theta_j,\quad j=2,\ldots,n\}$$ will be called \texttt{affine Lie superalgebra} and it will be denoted by $\mathrm{Aff}(2|n)$.
\subsection{Modules of weighted densities}

We define an action of $X_f$ on $C^\infty(S^{1|n})$ by
\begin{equation}\label{action}
L_{X_f}^\lambda(g)=X_f(g)+\lambda f'g,\quad \forall g\in C^\infty(S^{1|n})
\end{equation} and where $\lambda$ is an arbitrary number and $f'=\partial_xf$.
We thus obtain a familly of differential operators of order one on $C^\infty(S^{1|n})$, denoted by $L_{X_f}^\lambda$.
It is clear that
\[
[L_{X_f}^\lambda,L_{X_g}^\lambda]=L_{[X_f,X_g]}^\lambda=L_{X_{\{f,g\}}}^\lambda,
\]
where $\{f,g\}$ is the Lagrange bracket of the superfunctions $f$ and $g$. The correspondance $X_f\mapsto L_{X_f}^\lambda$ is a homomorphism of Lie superalgebras.
\begin{df}
The module $\mathcal{F}_\lambda(S^{1|n})$ of $\lambda$-contact densities on $S^{1|n}$ is the space $C^\infty(S^{1|n})$ endowed with the action \eqref{action} of $\mathcal{K}(n)$. One can write any $\lambda$-contact density as $g\alpha^\lambda$ where $\alpha$ is a contact $1$-form on $S^{1|n}$ and for an arbitrary superfunction $g$.
\end{df}
We can also define another type of action $\mathbb{L}_X^\lambda$ of an element $X$ of $\mathrm{Vect}(S^{1|n})$ on the space of superfunctions $C^\infty(S^{1|n})$ as follow
\begin{equation}\label{Action}
\mathbb{L}_X^\lambda(g):=X(g)+\lambda \mathrm{div}(X)g,\forall g\in C^\infty(S^{1|n}),
\end{equation}
where $\mathrm{div}(X)=f'+\sum_{i=1}^n(-1)^{\widetilde{g^i}}\partial_{\theta_i}g^i$.
We can show that for all $X,Y\in\mathrm{Vect}(S^{1|n}) $, we have
\[
[\mathbb{L}_X^\lambda,\mathbb{L}_Y^\lambda]=\mathbb{L}_{[X,Y]}^\lambda
\] 
\begin{df}
The module $\mathrm{Ber}_\lambda(S^{1|n})$ of tensor densities of weight $\lambda\in\R$ on $S^{1|n}$ is the space of superfunctions $C^\infty(S^{1|n})$ endowed with the action \eqref{Action} of $\mathrm{Vect}(S^{1|n})$. One can write for $g\in C^\infty(S^{1|n})$, any $\lambda$-tensor density as $g|Dx|^\lambda$.
\end{df}
The following remark is important.
\begin{rmk}
If the superdimension $m=1-n$ is different to $-1$, then the application 
\begin{equation}\label{IntertwAction}
\varphi:\mathcal{F}_\lambda(S^{1|n})\to\mathrm{Ber}_{\frac{2\lambda}{m+1}}(S^{1|n}):g\alpha^\lambda\mapsto g\vert Dx\vert^{\frac{2\lambda}{m+1}}
\end{equation} is an isomorphism of $\mathcal{K}(n)$-modules.

To show that $\varphi$ is bijective, it sufficient to compute that the application $\varphi$ intertwines the action of $X_f$ on the spaces $\mathcal{F}_\lambda(S^{1|n})$ and $\mathrm{Ber}_{\frac{2\lambda}{m+1}}(S^{1|n})$, i.e.
\[
\varphi(L^\lambda_{X_f}(g\alpha^\lambda))=\mathbb{L}^\frac{2\lambda}{m+1}_{X_f}(\varphi(g\alpha^\lambda)).
\]
\end{rmk}

\subsection{Modules of Differential operators and symbols}
\begin{df}
If we denote $(q^I)=(x,\theta^i)$, $i\in [1,n]$ the coordinates system of Darboux on $S^{1|n}$, we call differential operator $D$ of order $k$ on $S^{1|n}$, the application which maps on $\mathcal{F}_\lambda(S^{1|n})$ to $\mathcal{F}_\mu(S^{1|n})$ and it can be written in coordinates by
\[D:f\alpha^{\lambda}\mapsto(\sum_{I:|I|\leq k}{A_I\partial_{q^I}f)\alpha^{\mu}  
},\]
where $I=(i_0,i_1,\ldots,i_{n})$ is a multi-index, $|I|=i_{0}+i_1+\ldots+i_{n}$
 and $A_I$ is a superfunction for all $I$.
\end{df}

More explicitly, a differential operator $D$ can be written in coordinates by
 \begin{equation}\label{opdiff}
 \sum_{I:|I|\leq k}{A_I(\partial_{x})^{i_0}(\partial_{\theta_1})^{i_{1}}\ldots(\partial_{\theta_n})^{i_{n}}}.
 \end{equation}
Because $\partial_{\theta_i}^2=0$, the exponents $i_{1},\ldots,i_{n}$ in the expression \eqref{opdiff} are at most equal to $1$.

The differential operators of order $0$ are simply the multiplication by $(\mu-\lambda)$-densities.
We define the space of differential operators as follow
  \[
\mathcal{D}_{\lambda\mu}(S^{1|n})=\bigcup_{k=0}^\infty\mathcal{D}_{\lambda\mu}^k(S^{1|n}).
\]
If $D\in \mathcal{D}_{\lambda\mu}(S^{1|n})$ and if $X\in\mathcal{K}(n)$, then the action of $X$ on $\mathcal{D}_{\lambda\mu}(S^{1|n})$ is defined by the Lie derivative $\mathcal{L}_X^{\lambda\mu}$ via the following supercommutator  

   \begin{equation}\label{optoraction}
 \mathcal{L}_X^{\lambda\mu}D=L_X^{\mu}\circ D-(-1)^{\tilde{X}\tilde{D}}D\circ L_X^{\lambda}
 \end{equation} where $L_X^\lambda$ and $L_X^\mu$ are defined in \eqref{action}.
  We have thus a structure of $\mathcal{K}(n)$-module  on the space $\mathcal{D}_{\lambda\mu}(S^{1|n})$. \\
  
  The graded space associated to the space $\mathcal{D}_{\lambda\mu}(S^{1|n})$ is called the space of symbols and denoted by $\mathcal{S}_\delta(S^{1|n})$. It is defined in the following way
  \[
  \mathcal{S}_\delta(S^{1|n}):=\mathcal{D}^k_{\lambda\mu}(S^{1|n})/\mathcal{D}^{k-1}_{\lambda\mu}(S^{1|n}),\quad \delta=\lambda-\mu.
  \]
 We define also the following surjective application on $\mathcal{D}^k_{\lambda\mu}(S^{1|n})$ 
   \[
\sigma_k:\mathcal{D}^k_{\lambda\mu}(S^{1|n})\to\mathcal{S}_\delta(S^{1|n}):D\mapsto[D],   
   \]
  where the notation $[D]$ means the equivalence class of $D$. The action of $\mathcal{K}(n)$ on $\mathcal{S}^k_\delta(S^{1|n})$, denoted by $L_X^\delta$ is given by
  \[
L_X^\delta(S)=[\mathcal{L}^{\lambda\mu}_X(D)],  
  \]where $S=[D]$ is an element of $\mathcal{S}^k_\delta(S^{1|n})$.\\
  
When $S^{1|n}$ is endowed with the standard contact $\alpha$, then the superderivations are spanned on $C^\infty(S^{1|n})$ by the field of Reeb $\partial_x$ and by the vector fields $\overline{D}_i$ where $i=1,\ldots,n$. This allows us to define on the space of differential operators a different filtration which induces another type of symbols.

\begin{df}
If $K=(i_1,\ldots,i_n)$ is the multi-index of length $|K|=i_1+\ldots+i_n$ and if $D$ is an differential operator of order $k$ acting between $\lambda$- and $\mu$-densities on $S^{1|n}$, then $D$ is of the form \begin{equation}\label{NORMALOPDIFF}
\sum_{K:c+|K|\leqslant k}A_{cK}\partial_{x}^{c}\overline{D}^K,
\end{equation}
where $A_{cK}$ is a superfunction and $\overline{D}^K=\overline{D}^{i_1}_1\cdots \overline{D}^{i_{n}}_{n}$.
\end{df} 
We can now define the differential operator with an order of Heisenberg $d$ where one has the order $d\in \frac{1}{2}\mathbb{N}$ whereas the order $k\in\mathbb{N}$.

\begin{df}
If $D$ is a differential operator of the form \eqref{NORMALOPDIFF}, then we say that $D$ has an order of Heisenberg equal to $d$ if $c+\frac{1}{2}|K|\leqslant d$ for all $c,K$. We denote by $\mathcal{H}^d_{\lambda\mu}(S^{1|n})$ the space of differential operators of order of Heisenberg equal to $d$ on $S^{1|n}$.
\end{df}
We can see that this space of differential operators is therefore filtered by the order of Heisenberg. Indeed, the total space $\mathcal{H}_{\lambda\mu}(S^{1|n})$ is the union of the spaces $\mathcal{H}^d_{\lambda\mu}(S^{1|n})$, i.e.
 \[
\mathcal{H}_{\lambda\mu}(S^{1|n})=\bigcup_{d\in\frac{1}{2}\N}\mathcal{H}^d_{\lambda\mu}(S^{1|n}). 
 \]
 Because $\mathcal{H}^d_{\lambda\mu}(S^{1|n})\subset\mathcal{H}^{d+\frac{1}{2}}_{\lambda\mu}(S^{1|n})$, we have the following inclusions 

\[
 \mathcal{H}^0_{\lambda\mu}(S^{1|n})\subset\mathcal{H}^{\frac{1}{2}}_{\lambda\mu}(S^{1|n})\subset\mathcal{H}^1_{\lambda\mu}(S^{1|n})\subset\cdots\subset\mathcal{H}^d_{\lambda\mu}(S^{1|n})\subset\mathcal{H}^{d+\frac{1}{2}}_{\lambda\mu}(S^{1|n})\subset\cdots
 \] for all $d\in\frac{1}{2}\N$.

\begin{df}
The graded space associated to the space $\mathcal{H}^d_{\lambda\mu}(S^{1|n})$
is denoted by $\mathcal{P}_\delta(S^{1|n})$. We have therefore
\begin{equation}\label{PD}
\mathcal{P}_{\delta}(S^{1|n}):=\bigoplus_{d\in\frac{1}{2}\N} \mathcal{P}_{\delta}^{d}(S^{1|n}):=\sum_{d\in\frac{1}{2}\N} \mathcal{H}_{\lambda\mu}^{d}(S^{1|n})/\mathcal{H}_{\lambda\mu}^{d-\frac{1}{2}}(S^{1|n}),
\end{equation} where $\delta=\mu-\lambda$.
\end{df}

The canonical projection defines the application $h\sigma$ called Heisenberg symbol map as follow
\[{\rm h}\sigma:\mathcal{H}_{\lambda\mu}^{d}(S^{1|n})\to\mathcal{P}_{\delta}^{d}(S^{1|n}):D\mapsto [D],\]
where $[D]$ means the equivalence class of $D$ in the quotient $\mathcal{H}_{\lambda\mu}^{d}(S^{1|n})/\mathcal{H}_{\lambda\mu}^{d-\frac{1}{2}}(S^{1|n})$.

\subsection{Modules of fine symbols}
 We generalize in super case, the model used in the even case by C.Conley and V.Ovsienko in \cite{CoOv12}.

\begin{df}
We define a bifiltration on $\mathcal{D}_{\lambda\mu}(S^{1|n})$ by
\[
\mathcal{D}_{\lambda\mu}^{k,d}(S^{1|n}):=\mathcal{D}_{\lambda\mu}^{k}(S^{1|n})\cap\mathcal{H}_{\lambda\mu}^{d}(S^{1|n}). 
\] 
The bigraded space $\Sigma_\delta(S^{1|n})$ associated to the bifiltration $\mathcal{D}_{\lambda\mu}^{k,d}(S^{1|n})$ is defined by
\begin{equation}\label{BIGRADED}
\Sigma_{\delta}(S^{1|n})=\bigoplus_{k=0}^\infty\bigoplus_{d\in\frac{1}{2}\N}\Sigma_{\delta}^{k,d}(S^{1|n})=\bigoplus_{k=0}^\infty\bigoplus_{d\in\frac{1}{2}\N} \mathcal{D}_{\lambda\mu}^{k,d}(S^{1|n})/(\mathcal{D}_{\lambda\mu}^{k-1,d}(S^{1|n})+\mathcal{D}_{\lambda\mu}^{k,d-\frac{1}{2}}(S^{1|n})).
\end{equation}
The elements of $\Sigma_{\delta}(S^{1|n})$ are called fine symbols.
\end{df}
We define accordingly the fine symbol map by
\[{\rm f}\sigma_{k,d}:\mathcal{D}_{\lambda\mu}^{k,d}(S^{1|n})\to\Sigma_{\delta}^{k,d}(S^{1|n}):D\mapsto [D],\] where the bracket means the equivalence class of $D$
in $\mathcal{D}_{\lambda\mu}^{k,d}(S^{1|n})/(\mathcal{D}_{\lambda\mu}^{k-1,d}(S^{1|n})+\mathcal{D}_{\lambda\mu}^{k,d-\frac{1}{2}}(S^{1|n}))$.
To justify the terminology of fine symbol, we refer to \cite{CoOv12}.

If we define an isomorphism between $\mathcal{P}_\delta^d(S^{1|n})$ and a sub-space of $\mathcal{F}_\delta\otimes\mathrm{Pol}(T^*S^{1|n})$ as follow
\begin{equation}\label{ISOMOR_POLYN0}
\varphi:\sum_{\substack{c,K\\c+\frac{1}{2}|K|\leqslant d}}{A_{cK}\partial_{x}^{c}\overline{D}^K}+\mathcal{H}_{\lambda,\mu}^{d-\frac{1}{2}}(S^{1|n})\mapsto\sum_{\substack{c,K\\c+\frac{1}{2}|K|=d}}{A_{cK}\alpha^\delta\zeta^{c}\gamma^K}
\end{equation}

where $\zeta$ and $\gamma$ are respectively the moments associated to the vector fields $\partial_x$ and $\overline{D}_i$, then we can see that the space $\Sigma^{k,d}(S^{1|n})$ is isomorphic to a sub-space of $\mathcal{F}_\delta\otimes\mathrm{Pol}(T^*S^{1|n})$ in the following way

\begin{equation}\label{ISOMOR_POLYN02}
\varphi:\sum_{\substack{c,K\\c+\frac{1}{2}|K|\leqslant d\\c+|K|\leqslant k}}A_{cK}{\partial_{x}^{c}\bar{D}^K}+(\mathcal{D}_{\lambda,\mu}^{k,d-\frac{1}{2}}(S^{1|n})+\mathcal{D}_{\lambda,\mu}^{k-1,d}(S^{1|n}))\mapsto
\sum_{\substack{c,K\\c+\frac{1}{2}|K|=d\\c+|K|=k}}A_{cK}{\alpha^\delta\zeta^{c}\gamma^K}.
\end{equation}
One can see that the space of symbols $\mathcal{P}_\delta(S^{1|n})$ can be expressed with the direct sum of $\Sigma^{k,d}_\delta(S^{1|n})$:
\[
\mathcal{P}_\delta(S^{1|n})=\bigoplus_{d\in\frac{1}{2}\mathbb{N}}\mathcal{H}^d(S^{1|n})/\mathcal{H}^{d-\frac{1}{2}}(S^{1|n})=\bigoplus_{d\in\frac{1}{2}\mathbb{N}}\bigoplus_{k=\lceil d\rceil}^{2d}\Sigma_\delta^{k,d}(S^{1|n})
\]
where $\lceil x\rceil:=\mathrm{inf}\{n\in\N:n\geqslant x\}$.

The following theorem gives the explicit formulas of the actions of the Lie superalgebra $\spo(2|n)$ on the spaces 
$\mathcal{S}_{\delta}(S^{1|n})$ and $\mathcal{P}_\delta(S^{1|n})$ and on the fine symbol space $\Sigma_\delta(S^{1|n})$.

\begin{theorem}\label{ActionFormulas}
If $X_f\in \spo(2|n)$ and if we denote by $L_{X_f}^{\mathcal{P}}$ (resp. $L_{X_f}^{\Sigma}$, resp. $L_{X_f}^{\mathcal{S}}$ ) the actions of $X_f$ on $\mathcal{P}_{\delta}(S^{1|n})$ (resp. $\Sigma_{\delta}(S^{1|n})$; resp. $\mathcal{S}_{\delta}(S^{1|n})$) then these actions are given 
\begin{enumerate}
\item[i)] firstly by $
L^\Sigma_{X_f}=f\partial_x+\partial_x(f)\left(\delta-\zeta\partial_\zeta\right)
-\frac{1}{2}(-1)^{\tilde{f}}\sum_{i=1}^n\overline{D}_i(f)\overline{D}_i
+\frac{1}{2}\sum_{j,k=1}^n\overline{D}_j\overline{D}_k(f)\gamma_k
\partial_{\gamma_j}.
$
\item[ii)] secondly by $L_{X_f}^{\mathcal{P}} = L_{X_f}^{\Sigma}$, 

\item[iii)] and finally by $
L^S_{X_f}= L^\Sigma_{X_f}+\frac{1}{2}(-1)^{\tilde{f}}\sum_{i=1}^n\overline{D}_i(f')\gamma_i\partial_{\zeta}, 
$
\end{enumerate}
where the notations $\partial_x$ and $\overline{D}_i$ denote the action of the vector fields $\partial_x$ and $\overline{D}_i$ on the coefficients of the symbol.
\end{theorem}
\begin{proof}
The spaces $\Sigma(S^{1|n}):=\bigcup_{\delta\in\mathbb{R}}\Sigma_{\delta}(S^{1|n})$, $\mathcal{P}(S^{1|n}):=\bigcup_{\delta\in\mathbb{R}}\mathcal{P}_{\delta}(S^{1|n})$ and $\mathcal{S}(S^{1|n}):=\bigcup_{\delta\in\mathbb{R}}\mathcal{S}_{\delta}(S^{1|n})$ are actually algebras for the canonical product of symbols. We can consider the operators  $\tilde{L}_{X_f}^{\Sigma}$, $\tilde{L}_{X_f}^{\mathcal{P}}$ and $\tilde{L}_{X_f}^{\mathcal{S}}$  acting respectively on $\Sigma(S^{1|n})$, $\mathcal{P}(S^{1|n})$ and $\mathcal{S}(S^{1|n})$ whose the restrictions on the spaces $\Sigma_{\delta}(S^{1|n})$, $\mathcal{P}_{\delta}(S^{1|n})$ and $\mathcal{S}_{\delta}(S^{1|n})$ are given by $L_{X_f}^{\Sigma}$, $L_{X_f}^{\mathcal{P}}$ and $L_{X_f}^{\mathcal{S}}$ for all $\delta\in\mathbb{R}$. The operators $\tilde{L}_{X_f}^{\Sigma}$, $\tilde{L}_{X_f}^{\mathcal{P}}$ and $\tilde{L}_{X_f}^{\mathcal{S}}$ are actually derivations of the spaces $\Sigma(S^{1|n})$, $\mathcal{P}(S^{1|n})$ and $\mathcal{S}(S^{1|n})$.
We can compute the actions $L_{X_f}^\Sigma$ and $L_{X_f}^\mathcal{S}$ on the generators $\zeta$, $\gamma_i$ and $g\alpha^\delta$ of the spaces $\Sigma_\delta(M)$ and $\mathcal{S}_\delta(M)$.\\
If the application of $L_{X_f}^\Sigma$ and $L_{X_f}^\mathcal{S}$ on those generators coincide with the actions of the second members of equations $(i)$ and $(iii)$ of the theorem \ref{ActionFormulas} on the same generators and thanks that the right members of those equations are the derivation operators, then the equations $(i)$ and $(iii)$ of the theorem \ref{ActionFormulas} are verified.\\
Thanks to the isomorphism $\varphi$ defined by \eqref{ISOMOR_POLYN0}, we compute the operators $L_{X_f}^\Sigma$ and $L_{X_f}^\mathcal{S}$ on the generators $\zeta, \gamma_i$ and $g\alpha^\delta$ of the spaces $\Sigma_\delta(M)$ and $\mathcal{S}_\delta(M)$ by using  respectively the Lie derivative of differential operators $\partial_x,\overline{D}_i$ and $g\alpha^\delta$.\\
We obtain 
\[
L^\Sigma_{X_f}(\overline{D}_i)=[X_f,\overline{D}_i]+\delta f'\overline{D}_i
\]
Since the commutator $[\overline{D}_i,\overline{D}_s]$ is equal to $-2\partial_z$ and using the isomorphism $\varphi$ given by \eqref{ISOMOR_POLYN0}, we obtain
\begin{equation}\label{Lie1}
L_{X_f}^\Sigma(\gamma_i)=\left(\delta f'\Id+\frac{1}{2}\sum_{j,k=1}^n\overline{D}_j\overline{D}_k(f)\gamma_k\partial_{\gamma_j}\right)(\gamma_i).
\end{equation}
If we compute $L_{X_f}^\Sigma(\partial_x)$ and using the isomorphism $\varphi$ given by \eqref{ISOMOR_POLYN0}, we obtain
\begin{equation}\label{Lie2}
L_{X_f}^\Sigma(\zeta)=\left(\delta-1\right)f'\Id(\zeta),
\end{equation}
and finally we obtain in the same conditions
\begin{equation}\label{Lie3}
L_{X_f}^\Sigma(g\alpha^\delta)=\left(f\partial_x+\delta f'\Id-(-1)^{\tilde{f}}\frac{1}{2}\sum_{i=1}^n\overline{D}_i(f)\overline{D}_i\right)(g\alpha^\delta).
\end{equation}
We can see that if we restrict the formula given by $(i)$ in the theorem \ref{ActionFormulas} to the generators $\overline{D}_i$, $\partial_x$ and $g\alpha^\delta$, we obtain respectively the formulas (\ref{Lie1}), (\ref{Lie2}) et (\ref{Lie3}). The proof is also the same on the formula given by $(iii)$ in the theorem \ref{ActionFormulas}.

\end{proof}

\begin{rmk}
By definition of action, the applications ${\rm f}\sigma_{k,d}$ and $\sigma_k$ are $\mathcal{K}(n)$-equivariant, i.e.
\[
L_{X_f}^{\mathcal{S}}\circ\sigma_k=\sigma_k\circ\mathcal{L}_{X_f}^{\lambda\mu}\quad\mbox{on}\;\;\mathcal{D}_{\lambda\mu}^k(S^{1|n}),
\] and
\[
L_{X_f}^{\Sigma}\circ {\rm f}\sigma_{k,d}={\rm f}\sigma_{k,d}\circ \mathcal{L}_{X_f}^{\lambda\mu}\quad\mbox{on}\;\; \mathcal{D}^{k,d}_{\lambda\mu}(S^{1|n}).
\] 
\end{rmk}

\subsection{The affine quantization map and the map $\gamma$}
Here we use the tools of the construction of the quantization in the purely even situation to build the quantization in the super situation. We begin by the  classical definition of a quantization.
\begin{df}
If $\mathfrak{g}$ is a Lie superalgebra, we call a $\mathfrak{g}$-equivariant quantization on $S^{1|n}$, a linear application 
$$ Q:\mathcal{S}_\delta(S^{1|n})\to\mathcal{D}_{\lambda\mu}(S^{1|n}) $$ such that, for every $X\in \mathfrak{g}$, one has 
$$\mathcal{L}^{\lambda\mu}_X\circ Q=Q\circ L_X^\delta $$ and satisfy the condition
\[\quad\forall k\in\N,\quad\forall S\in\mathcal{S}^k_\delta(S^{1|n}),\quad \sigma_k(Q(S))=S.\]
\end{df}
Let now build (in our case) a linear even bijection $Q_{\mathrm{Aff}}$ from the space of symbols to the space of differential operators. This quantization is defined as the inverse of the total symbol map $\sigma_{\mathrm{Aff}}$ whose restriction to $\mathcal{D}_{\lambda\mu}^k$ is given in the following definition.
\begin{df}
For all $k\in\mathbb{N}$, one has
\[\sigma_{\mathrm{Aff}}:\mathcal{D}^k_{\lambda\mu}(S^{1|n}) \to \oplus_{l=0}^k\mathcal{S}_\delta^l(S^{1|n}):D\mapsto \sigma_{\mathrm{Aff}}(D),\] 
where, if $D$ is given by \eqref{opdiff}, therefore $\sigma_{\mathrm{Aff}}(D)$ is defined by
\[ 
 \sum_{I:|I|\leq k}{\left[A_I(\partial_{x})^{i_0}(\partial_{\theta_1})^{i_{1}}\ldots(\partial_{\theta_n})^{i_{n}}\right]_
 {|I|}},
\]
where $[\cdot]_r$ means the equivalence class in $\mathcal{D}^r_{\lambda\mu}(S^{1|n})/\mathcal{D}^{r-1}_{\lambda\mu}(S^{1|n})$.
\end{df}
The affine quantization intertwines the actions of the affine superalgebra.\\
As done in \cite{MatRad11,MelNibRad13}, using the affine quantization map, we can endow the space of symbols with a structure of representation of $\mathrm{Vect}(S^{1|n})$, isomorphic to $\mathcal{D}_{\lambda\mu}$. Explicitly, we set
\[\mathcal{L}_{X_f}S:=Q^{-1}_{\mathrm{Aff}}\circ\mathcal{L}^{\lambda\mu}_{X_f}\circ Q_{\mathrm{Aff}}(S), \quad\forall S\in\mathcal{S}_\delta(S^{1|n}),X_f\in \mathfrak{spo}(2|n).\]
In order to measure the difference between the representations $(\mathcal{S}_\delta,L_{X_f})$ and $(\mathcal{S}_\delta,\mathcal{L}_{X_f})$, we define the map
\begin{equation}\label{Affinemap}
\gamma:\mathfrak{spo}(2|n)\to\mathfrak{gl}(\mathcal{S}_\delta,\mathcal{S}_\delta):X\mapsto \gamma(X_f)=\mathcal{L}_{X_f}-L^\delta_{X_f}.
\end{equation}
This map can be computed in coordinates and it has the same properties as in the classical situation. As consequence of a result of P. Mathonet and F. Radoux in \cite{MatRad11}, we have the following result
\begin{prop}
The map $\gamma$ vanishes on the constant vector fields and linear vector fields. Moreover, for every quadratic vector field $X_f$ and $k\in\N$, the restriction of $\gamma(X_f)$ to $\mathcal{S}_\delta^k$ has values in $\mathcal{S}_\delta^{k-1}$ and is a differential operator of order zero and parity $\widetilde{X_f}$ with constant coefficients.
\end{prop}

We are now in position to define the notion of fine $\spo(2|n)$-equivariant quantization on $S^{1|n}$.
\begin{df}
A fine $\spo(2|n)$-equivariant quantization is a linear bijection $Q$ between $\mathcal{P}_\delta(S^{1|n})$ and $\mathcal{D}_{\lambda\mu}(S^{1|n})$ such that its restriction $Q\vert_{\mathcal{P}_\delta^d}:\mathcal{P}_\delta^d(S^{1|n})\to\mathcal{D}_{\lambda\mu}^{k,d}(S^{1|n})$ is such that the following application 

\[
h\sigma^d\circ Q\vert_{\mathcal{P}^d_\delta}:\mathcal{P}^d_\delta(S^{1|n})\to\mathcal{P}^d_\delta(S^{1|n})
\]is the identity.\\
We say that the fine quantization $Q$ is  $\spo(2|n)$-equivariant if it intertwines the actions of $\spo(2|n)$ on $\mathcal{P}_\delta(S^{1|n})$ and on $\mathcal{D}_{\lambda\mu}(S^{1|n})$.
\end{df}
\section{Construction of the fine $\spo(2|n)$-equivariant quantization on $S^{1|n}$}
In order to construct the fine $\spo(2|n)$-equivariant quantization we will need to adopt the tools used in \cite{CoOv12} in the even case. It is mainly to construct the equivariant application between $\mathcal{P}_\delta(S^{1|n})$ and $\mathcal{S}_\delta(S^{1|n})$.
We begin with the main results given in the following lemmas.
\begin{lem}\label{tools}
If $i\in\{1,\ldots,n\}$ then one has the commutator
\[
[X_{x^2},X_{\theta_i}]=-X_{x\theta_i}.
\]
\end{lem}
\begin{proof}
It is clear that the result is obtained by using the Lagrange formula.
\end{proof}
In particular, the formulas obtained in the theorem \ref{ActionFormulas} becomes with $X_{x^2}$ and $X_x$ respectively as follow
\begin{equation}\label{Commutator1}
L^\Sigma_{X_{x^2}}=x^2\partial_{x^2}+2x(\delta-\zeta\partial_\zeta)+x\theta_i\overline{D}_i-x\gamma_i\partial_{\gamma_i}
\end{equation} and
\begin{equation}\label{Commutator2}
xL^\Sigma_{X_{x}}=x^2\partial_{x}+x(\delta-\zeta\partial_\zeta)+\frac{1}{2}x\theta_i\overline{D}_i-\frac{1}{2}x\gamma_i\partial_{\gamma_i}.
\end{equation}
The formulas \eqref{Commutator1} and \eqref{Commutator2} and the lemma \ref{tools} are very important in the computations.
\begin{df}\label{Delta}
We define now the following operator on the space of fine symbols $\Sigma^{k,d}(S^{1|n})$ by setting
\[
\Delta:\Sigma^{k,d}(S^{1|n})\to\Sigma^{k,d-\frac{1}{2}}(S^{1|n}): S\mapsto \sum_{i=1}^n\gamma_i\overline{D}_i\partial_\zeta S.
\]
\end{df}

In the following, we use the Einstein convention on the sum of the repeated indexes. The following commutators facilitate the computations.
\begin{lem}\label{comm1}
\begin{enumerate}
\item[(i)] The commutator of $\Delta$ and $\zeta\partial_\zeta$ is equals to $\Delta$, i.e. $[\Delta,\zeta\partial_\zeta]=\Delta$,
\item[(ii)] $[\Delta,x]=\theta_i\gamma_i\partial_\zeta$,
\item[(iii)] The operator $\Delta$ intertwines the actions of the affine superalgebra $\mathrm{Aff}(2|n)$, i.e. $[\Delta, L_{X_f}]=0,\forall X_f\in\mathrm{Aff}(2|n)$.
\item[(iv)] $[\Delta,\gamma_i]=0$.
\end{enumerate}
\end{lem}
These results are obtained by a simple computation. We have also the following commutators.
\begin{lem}\label{comm2}
\item[(i)] $[\Delta,x^2]=2x\theta^i\gamma_i\partial_\zeta$,
\item[(ii)] $[\Delta,\theta^l\gamma_l]=0$,
\item[(iii)] $[\Delta,\theta^j\partial_{\gamma_j}]=\gamma_j\partial_{\gamma_j}\partial_\zeta-\theta^i\overline{D}_i\partial_\zeta$.
\end{lem}
It sufficient to use the elementary results of the lemma \ref{comm1}.
\subsection{$\spo(2|n)$-equivariant application between $\mathcal{P}_\delta(S^{1|n})$ and $\mathcal{S}_\delta(S^{1|n})$}

As done in the even case (see e.g in \cite{CoOv12}), we define in this section the application $\mathcal{SQ}$ between $\mathcal{P}_\delta(S^{1|n})$ and $\mathcal{S}_\delta(S^{1|n})$ and we show that it intertwines the actions of the Lie superalgebra $\spo(2|n)$ on the spaces $\mathcal{P}_\delta(S^{1|n})$ and $\mathcal{S}_\delta(S^{1|n})$. 
\begin{theorem}\label{main}
If $\delta$ is not contained in the set \[
I_\delta:= \bigcup_{d=0}^\infty\{\frac{2c-j}{2},\quad 0\leqslant j\leqslant 2d-1,\quad 0\leqslant c\leqslant d\}
\] then the application $\mathcal{SQ}:\mathcal{P}_\delta(S^{1|n})\to \mathcal{S}_\delta(S^{1|n})$ defined by
\[
S\mapsto\left(\sum_{a=0}^\infty{\Delta^a\frac{1}{2^aa!}\prod_{j=0}^{a-1}{\frac{1}{\left(c-\delta-\frac{1}{2}j\right)}}}\right)S
\]where $c$ is the degree of $S$ in $\zeta$, is a $\spo(2|n)$-equivariant application between $\mathcal{P}_\delta(S^{1|n})$ and $\mathcal{S}_\delta(S^{1|n})$ if we set  $\prod_{j=0}^{-1}{\frac{1}{\left(c-\delta-\frac{1}{2}j\right)}}=1$. 
\end{theorem}
\begin{proof}
It is clear that the application $\mathcal{SQ}$ intertwines the actions of $\mathrm{Aff}(2|n)$. Because of the commutator $[X_{x^2},X_{\theta_i}]$ given in the lemma \ref{tools}, it is clear that when we prove that $\mathcal{SQ}$ intertwines the action of $X_{x^2}$ on the spaces $\mathcal{P}_\delta(S^{1|n})$ and $ \mathcal{S}_\delta(S^{1|n})$, then we prove that the application $\mathcal{SQ}$ is $\spo(2|n)$-equivariant between $\mathcal{P}_\delta(S^{1|n})$ and  $\mathcal{S}_\delta(S^{1|n})$.
The application of the following form $\sum_{a=0}^\infty{\Delta^aC_a}$ where $C_a$ is a constant for all integer $a$, is $\spo(2|n)$-equivariant between  $\mathcal{P}_\delta(S^{1|n})$ and  $\mathcal{S}_\delta(S^{1|n})$ if and only if we have the condition
\[
L^S_{X_{x^2}}\circ\sum_{a=0}^\infty{\Delta^aC_a}=
\sum_{a=0}^\infty{\Delta^aC_a}\circ L^\Sigma_{X_{x^2}}
\]
We can also write the equation as follow
\[
(L^S_{X_{x^2}}-L^\Sigma_{X_{x^2}})\circ\sum_{a=0}^\infty{\Delta^aC_a}+ L^\Sigma_{X_{x^2}}\circ
\sum_{a=0}^\infty{\Delta^aC_a}=\sum_{a=0}^\infty{\Delta^aC_a}\circ L^\Sigma_{X_{x^2}}.
\]
By using the third formula of theorem \ref{ActionFormulas}, we can see that $L^S_{X_{x^2}}-L^\Sigma_{X_{x^2}}=-\theta^s\gamma_s\partial_{\zeta}$. Therefore, the previous equality becomes
\begin{equation}\label{demo}
\theta^s\gamma_s\partial_{\zeta}\circ\sum_{a=0}^\infty{\Delta^aC_a}+
\sum_{a=0}^\infty[L^\Sigma_{X_{x^2}},\Delta^a]C_a=0.
\end{equation}
We need to know the commutator $[L^\Sigma_{X_{x^2}},\Delta^a]$ which we compute in the following lemma.
\begin{lem}\label{deltaOper}
For all non vanishing integer $a$, we have
\[
[L^\Sigma_{X_{x^2}},\Delta^a]=2a\Delta^{a-1}\theta^s\gamma_s\partial_\zeta(\zeta\partial_\zeta-\delta-\frac{1}{2}(a-1)).
\] 
\end{lem}
\begin{proof}
We prove this by induction. We suppose that the formula is correct for $a=1$. Indeed, because of the formulas \eqref{Commutator1} and \eqref{Commutator2}, we can write \[
[\Delta,L^\Sigma_{X_{x^2}}]=[\Delta,2xL^\Sigma_{X_x}]-[\Delta,x^2\partial_x].
\]
Since the operator $\Delta$ is $\mathrm{Aff}(2|n)$-equiavriant, it sufficient to compute the commutators $[\Delta,x]$ and $[\Delta,x^2\partial_x]$ by using the elementary results of lemma\ref{comm1} and lemma \ref{comm2}. We show now that the formula is true for all $a\in\mathbb{N}$. We can write that the commutator $[L^\Sigma_{X_{x^2}},\Delta^a]$ is equals to
\[
[L^\Sigma_{X_{x^2}},\Delta^a]=[L^\Sigma_{X_{x^2}},\Delta^{a-1}]\Delta+\Delta^{a-1}[L^\Sigma_{X_{x^2}},\Delta].
\]
If we suppose true the formula given in the lemma \ref{deltaOper} for $a-1$, the expression $[L^\Sigma_{X_{x^2}},\Delta^{a-1}]\Delta+\Delta^{a-1}[L^\Sigma_{X_{x^2}},\Delta]$ in the second member is equals to
\[
2(a-1)\Delta^{a-2}\theta^s\gamma_s\partial_\zeta\left(\zeta\partial_\zeta-\delta-\frac{1}{2}(a-2)\right)\Delta+2\Delta^{a-1}\theta^s\gamma_s\partial_\zeta(\zeta\partial_\zeta-\delta).
\]
By using the fact that $[\delta,\zeta\partial_\zeta]=\Delta$ and that $[\Delta,\theta^s\gamma_s]=0$, the commutator $[L^\Sigma_{X_{x^2}},\Delta^a]$ becomes
\[
2\Delta^{a-1}\theta^s\gamma_s\partial_\zeta\left((a-1)(\zeta\partial_\zeta-\delta)
-\frac{1}{2}a(a-1)+(\zeta\partial_\zeta-\delta)\right).
\]
A simple computation gives the result.
\end{proof}
We return now to the proof of the theorem. Using the formula of lemme \ref{deltaOper}, the expression \eqref{demo} can be written as follow
\[
-\theta^s\gamma_s\partial_\zeta\circ\sum_{a=0}^\infty\Delta^aC_a+2\sum_{a=1}^\infty a\Delta^{a-1}\theta^s\gamma_s\partial_\zeta\left(\zeta\partial_\zeta-\delta-\frac{1}{2}(a-1)\right)C_a=0,
\] or also
\[
\sum_{a=1}^\infty \theta^s\gamma_s\partial_\zeta\Delta^{a-1}\left(-C_{a-1}+2a\left(\zeta\partial_\zeta-\delta-\frac{1}{2}(a-1)\right)C_a\right)=0.
\]
An application of the form $\sum_{a=0}^\infty \Delta^a C_a$ is $\spo(2|n)$-equivariant if and only if the constant $C_a$ of the theorem \ref{main} are given by the following formula
\[
C_a=\frac{C_{a-1}}{2a\left(c-\delta-\frac{1}{2}(a-1)\right)}
\] for all $a\in\N_0$.
\end{proof}
\subsection{Fine $\spo(2|n)$-equivariant quantization}

In this section, we show that in order to construct the fine $\spo(2|n)$-equivariant quantization, we need to use the $\pgl(p+1|q)$-equivariant quantization constructed by P.Mathonet and F. Radoux in \cite{MatRad11}. We begin here by their main results on $\mathbb{R}^{p|q}$ for all $p,q\in\mathbb{N}$.\\
If $p+1\neq q$, then the projective Lie superalgebra $\pgl(p+1|q)$ is isomorphic to the Lie superalgebra $\mathfrak{sl}(p+1|p)$. The set of critical values of the $\mathfrak{sl}(p+1|q)$-equivariant application between $\mathcal{S}_\delta$ and $\mathcal{D}_{\lambda\mu}$ on $\mathbb{R}^{p|q}$ is given by the formula
 \begin{equation}\label{PGL_CRIT}
\mathfrak{C}=\cup_{k=1}^\infty\mathfrak{C}_k, \quad\mbox{where}\quad \mathfrak{C}_k=\{\frac{2k-i+p-q}{p-q+1}:i=1,\cdots,k\}.  
 \end{equation} 
The following theorem gives the main result of $\mathfrak{sl}(p+1|q)$-equivariant quantization on $\mathbb{R}^{p|q}$.
\begin{theorem}\label{PR}
If $\delta$ is not critical value, then the application $Q:\mathcal{S}_{\delta}\to \D_{\lambda\mu}$ defined by 
\[  
Q(S)(f)=\sum_{r=0}^k{C_{k,r}Q_{\mathrm{Aff}}(\mathrm{div}^rS)(f)},\quad\forall S\in\mathcal{S}_{\delta}^k
\] is the unique $\mathfrak{sl}(p+1|q)$-equivariant quantization on $\R^{p|q}$ if \[
C_{k,r}=\frac{\prod_{j=1}^r((p-q+1)\lambda+k-j)}{r!\prod_{j=1}^r(p-q+2k-j-(p-q+1)\delta)},\quad\forall r\geqslant 1, \quad C_{k,0}=1.
\]
\end{theorem}
Here the application "$\mathrm{div}$" was extended to the symbols of degree $k$ as follow:
\[
\mathrm{div}:\mathcal{S}^k_{\delta}\to\mathcal{S}^{k-1}_{\delta}:S\mapsto \sum_{j=1}^{p+q}{(-1)^{\tilde{y_j}}i(\epsilon^j)\partial_{y^j}S}, 
 \] where $\epsilon^j$ is the $j$-th line vector of the base of $(\R^{p|q})^* \cong\mathfrak{g}_1$ and where $S$ is considered as an element of $\mathcal{F}_\delta\otimes\mathrm{Pol}^k(T^*\R^{p|q})$ and $p=2l+1$ for all $l\in\mathbb{N}$.\\
 
When the superspace $\R^{p|q}$ is endowed with the standard contact structure, i.e. when $p=2l+1$ for all $l\in\mathbb{N}$, the $\mathfrak{sl}(p+1|q)$-equivariant quantization induces a $\spo(p+1|q)$-equivariant quantization.
The set of critical values given by the formula \eqref{PGL_CRIT} of the induced  $\spo(p+1|q)$-equivariant quantization becomes
\begin{equation}\label{CRITIQUE}
\mathfrak{C}'=\cup_{k=1}^\infty\mathfrak{C}'_k, \quad\mbox{where}\quad \mathfrak{C}'_k=\{\frac{2k-i+m}{2}:i=1,\cdots,k\}  
 \end{equation} where $m=p-q$ is the superdimension
and the theorem \ref{PR} takes the following form
\begin{theorem}
If $\delta$ is not critical value, then the application $Q:\mathcal{S}_{\delta}\to \D_{\lambda\mu}$ defined by 
\[  
Q^{\mathfrak{sl}}(S)(f)=\sum_{r=0}^k{C_{k,r}Q_{\mathrm{Aff}}(\mathrm{div}^rS)(f)},\quad\forall S\in\mathcal{S}_{\delta}^k
\] is the unique $\mathfrak{spo}(p+1|q)$-equivariant quantization on $\R^{p|q}$ if \[
C_{k,r}=\frac{\prod_{j=1}^r(2\lambda_c+k-j)}{r!\prod_{j=1}^r(m+2k-j-2\delta_c)},\quad\forall r\geqslant 1, \quad C_{k,0}=1
\] where the index $c$ denotes the contact weight.
\end{theorem}
In particular, when $p=1$ and $q=n$ we obtain the situation of the super circles $S^{1|n}$ by identifying the super circles $S^{1|n}$ with the super space $\R^{1|n}$.
We are now in position to construct the fine $\spo(2|n)$-equivariant quantization on $S^{1|n}$ by composing the applications $\mathcal{SQ}$ and $Q^{\mathfrak{sl}}$ in the following theorem.

\begin{theorem}\label{Existence}
If  $\delta\notin I_\delta\cup\mathsf{C}$ where $\mathsf{C}=\bigcup_{k\in\N}\mathfrak{C}'_k$ and if $Q^{\mathfrak{sl}}$ is the $\pgl(2|n)$-equivariant quantization, then the application $Q=Q^{\mathfrak{sl}}\circ\mathcal{SQ}$ is a fine $\spo(2|n)$-equivariant quantization on 
\[
\mathcal{P}_\delta(S^{1|n})=\bigoplus_{d\in\frac{1}{2}\N}\mathcal{P}^d_\delta(S^{1|n}).
\]
\end{theorem}
 \begin{proof}
  The sufficient condition of existence of $Q$ on $\mathcal{P}^d_\delta(S^{1|n})$ is simply the existence of $\mathcal{SQ}$ on $\mathcal{P}^d_\delta(S^{1|n})$ and the existence of $Q^{\mathfrak{sl}}$ on $\mathcal{SQ}(\mathcal{P}^d_\delta(S^{1|n}))$. The existence of $\mathcal{SQ}$ on $\mathcal{P}^d_\delta(S^{1|n})$ is assured by the fact that
  \[
  \delta\notin \{\frac{2c-j}{2},\quad 0\leqslant j\leqslant 2d-1,\quad 0\leqslant c\leqslant d\}
  \]
  whereas the existence of $Q^{\mathfrak{sl}}$ on $\mathcal{SQ}(\mathcal{P}^d_\delta(S^{1|n}))$ is assured by the fact that $\delta\notin \bigcup_{k=\lceil d \rceil}^{2d}\mathfrak{C}'_k$. Finally, the existence of $Q$ on $\mathcal{P}_\delta(S^{1|n})$ is assured by the fact that $\delta\notin I_\delta\cup\mathsf{C}$.
 \end{proof}
\section{Uniqueness of the fine $\spo(2|n)$-equivariant quantization}
The uniqueness of the fine $\spo(2|n)$-equivariant quantization is treated by comparing the Casimir operator $\mathcal{C}^\mathcal{P}$ on the fine symbols and the Casimir operator $C^\mathcal{D}$ on the differential operators.
\subsection{Casimir operators}

Let us recall the definition of the Casimir operator associated with a representation of a Lie superalgebra (see e.g \cite{MelNibRad13,MatRad11,Nib14,KacAdvances}).

\begin{df}\label{Casimir}
We consider a Lie superalgebra $\mathfrak{g}$ endowed with an even non-degenerate supersymmetric bilinear form $K$ and a representation $(V,\rho)$ of $\mathfrak{g}$. The Casimir operator $C_\rho$ of $\mathfrak{g}$ associated with $(V,\rho)$ is defined by 
\[
C_\rho=\sum_i\rho(u_i^*)\rho(u_i)
\] where $u_i$ and $u_i^*$ are $K$-dual bases of $\mathfrak{g}$, in the sense that $K(u_i,u_j^*)=\delta_{ij}$ for all $i,j$.
\end{df}
In the sequel, the bilinear form that we will use to define the Casimir operators of $\spo(2|n)$ will be the form $K$ defined in this way:
\begin{equation}\label{Form}
K(A,B)=2str(AB),\quad\forall A,B\in\spo(2|n).
\end{equation}

The following lemma will be important when we will compute the Casimir operator of $\spo(2|n)$.
\begin{lem}\label{dual}
The $K$-dual basis corresponding to the basis 
 \[
 X_x,\quad X_{x^2},\quad X_1,
 \quad X_{\theta_i},\quad X_{x\theta_i},\quad
  X_{\theta_{i-1}\theta_{j-1}}
 \]
 of $\mathfrak{spo}(2|n)$ is given by 
 \[
 X_x,\quad -\frac{1}{2}X_1,\quad -\frac{1}{2} X_{x^2},\quad -X_{x\theta_i},\quad X_{\theta_i},\quad X_{\theta_{i-1}\theta_{j-1}}.
 \]
\end{lem}
\begin{proof}
We use the correspondence between vector fields of $\spo(2|n)$ and the matrices established in \cite{MatRad11,Nib15,Nib14}.
\end{proof}
\subsection{Casimir operators on $\mathcal{S}_\delta$ and $\Sigma_\delta$}
The following result is proven by computation.
\begin{prop}
If $m=2l+1-n$ is the superdimension then 
\begin{enumerate}
\item the Casimir operator associated with the representation $(\mathcal{P}_\delta,L^\mathcal{P})$ is given by:
\begin{equation}\label{Finsymbolcas}
C^\mathcal{P}\vert_{\Sigma_\delta^{k,d}}=\alpha_{k,d} \Id
\end{equation}
 where $\alpha_{k,d}$  is a constant which depending with the degrees $k$ and $d$ of differential operators and it equals to:
\begin{multline}\label{FUNCTION_VALUE}
 \alpha_{k,d}=k^2+2(d^2-kd)+(k-d)(m-2\delta)\\
+\frac{1}{2}(2d-k)(m+1-4\delta)+\frac{1}{2}\delta(2\delta-m-1).
\end{multline}
\item the Casimir operator associated with the representation $(\mathcal{S}^k_\delta,L^S)$ is given by:
\begin{equation}\label{Normsymbolcas}
C^S=C^\Sigma+\frac{1}{2}\Delta,
\end{equation}
 where $\Delta$ is defined in definition \ref{Delta}.
\end{enumerate}
\end{prop}
\begin{proof}
First, we use the definition \ref{Casimir} and the formula \eqref{Form} and lemma\ref{dual} and we can write the casimir operator $C^\Sigma$ on fine symbols as follow
\begin{multline*}
-\frac{1}{2}L^\Sigma_{X_{x^2}}L^\Sigma_{X_1}-\sum_{i=1}^n{L^\Sigma_{X_{x\theta_i}}
L^\Sigma_{X_{\theta_i}}}+(L^\Sigma_{X_x})^2+\sum_{i,j}^n\sum_{i<j}{(L^\Sigma_{X_{\theta_{i-1}\theta_{j-1}}})^2}\\
+\sum_{j=1}^n{L^\Sigma_{X_{\theta_j}}L^\Sigma_{X_{x\theta_j}}}
-\frac{1}{2}L^\Sigma_{X_1}L^\Sigma_{X_{x^2}}.
\end{multline*}

 Secondly, we use the formula of $L_{X_f}^\Sigma$  given in the theorem\ref{ActionFormulas}. The fact that $C^\Sigma$ is seen as a differential operator written in $\partial_x$, $\partial_\zeta$ and in $\overline{D}_i$ and the fact that $C^\Sigma$ commutes with $\partial_x$ and that $\partial_x$ commutes with $\overline{D}_i$, the coefficients of the operator $C^\Sigma$ don't depend on the variable $x$. In the same sense, the operator $C^\Sigma$ commutes with $X_{\theta_i}$ and the coefficients of the operator $C^\Sigma$ don't depend on the variable $\theta_i$.\\
Because of the explicit formulas of the terms
 $L^\Sigma_{X_{x^2}}$ and $\sum_i{L^\Sigma_{X_{x\theta_i}}}$, the contributions of the terms $-\frac{1}{2}L^\Sigma_{X_{x^2}}L^\Sigma_{X_1}$ and $\sum_{i=1}^n{L^\Sigma_{X_{x\theta_i}}L^\Sigma_{X_{\theta_i}}}$ are equal to zero. The contribution of the term $-\frac{1}{2}L^\Sigma_{X_1}L^\Sigma_{X_{x^2}}$ is the same that its given by \[-\frac{1}{2}L^\Sigma_{[X_1,X_{x^2}]}=-\frac{1}{2}L^\Sigma_{X_{\{1,x^2\}}}=-L^\Sigma_{X_x}=-\left(
\delta-\zeta\partial_\zeta-\frac{1}{2}\sum_{k}{\gamma_k\partial_{\gamma_k}}  \right).\]
The contributions of the other terms of $C^\Sigma$ are obtained by computation. The sum of the all contributions of the Casimir operator $C^\Sigma$ is polynomial in Euler operators corresponding in $\zeta$ and $\gamma_i$ coordinates. It is therefore easy to see that the restriction of $C^\mathcal{P}$ to $\Sigma^{k,d}$ is equal to $\alpha_{k,d}\mathrm{Id}$. The proof of the formula \eqref{Normsymbolcas} is similar. We use the formula of $L_{X_f}^S$ given in the theorem\eqref{ActionFormulas}.
\end{proof}

\subsection{Casimir operator on $\mathcal{D}_{\lambda\mu}$}
In order to compute the Casimir operator on  $\mathcal{D}_{\lambda\mu}$ we use the tools developed by P. Mathonet and F. Radoux in \cite{MatRad11}. These tools which are axed to a interior product of a row vector and a symmetric tensor. This allow us to adapt their results to our case. We recall first this definition given in \cite{MatRad11}. 
\begin{df}\label{Int}
The interior product of $h\in(\R^{2l+1|n})^*$ in a symmetric tensor $h_1\vee\ldots\vee h_k$ $(h_1,\ldots, h_k\in\R^{2l+1|n})$ is defined by 
 \[i(h)h_1\vee\ldots\vee h_k=\sum_{j=1}^k{(-1)^{\tilde{h}(\sum_{r=1}^{j-1}{\tilde{h_r})}}\langle h,h_j\rangle h_1\vee\ldots \hat{j}\ldots\vee h_k},\] where $\langle h,x\rangle$ denotes the standard matrix multiplication of the row $h$ by the column $x$.
\end{df}
This means that we extend the pairing of $(\R^{2l+1|n})^*$ and $\R^{2l+1|n}$ as a derivation of the symmetric product. We may also extend it to $\mathcal{S}_\delta$ by setting $i(h)u=0$ for $u\in\mathcal{F}_\delta$ (it isn't act on the coefficient of the polynomial) and by defining $i(h)$ as a differential operator of order zero and parity $\tilde{h}$.
We also recall (see \cite{MatRad11}) that if $h\in(\R^{2l+1|n})^*$ therefore $X^h$ is the quadratic projective vector field given by 
\[
X^h=\sum_{j=1}^{2l+1+n}{h_jt^j(-1)^{\tilde{j}}\mathcal{E}}
\]
 where $\mathcal{E}$ is the Euler vector field. This tools allows us to compute the operator $\gamma$ in simple terms.\\
 In particular, when $l=0$, we obtain the results of $S^{1|n}$. The following result expresses the operator $\gamma$.
\begin{lem}\label{Gamma2}
If $h$ is a quadratic contact projective vector field and if $h'\in (\R^{1|n})^*$
is such that $X^{h'}=h$ therefore one has on $\mathcal{S}^k_\delta(S^{1|n})$
\begin{equation}
\gamma(h)=-\left(2\lambda+k-1\right)i(h'),
\end{equation} where $i(h')$ is the interior product from Definition\ref{Int}.
\end{lem}
\begin{proof}
The proof is similar as one done in details in \cite{MatRad11}.
\end{proof}
Using these tools we have the following result.
\begin{prop}\label{coucou}
Let $\varepsilon^k$ the dual vector of  the $k^{th}$ column vector of canonical basis of $\R^{1|n}$. We denote by $\zeta$(resp. $\eta_k$) the moment coordinate associated to the variable $x$ (resp. $\theta_k$). When we write a symbol in terms of the contact moment coordinates $\zeta$ and $\gamma_i$, then
\begin{enumerate}
\item  the operator $i(\varepsilon^1)$ is seen as the operator $\partial_\zeta-\theta_k\partial_{\gamma_k}$,
\item and the operator $i(\varepsilon^k)$, where $2\leqslant k\leqslant n+1$, is seen as the operator  $-\frac{1}{2}\partial_{\gamma_k}$.
\end{enumerate}
\end{prop}
\begin{proof}
The symbol $\zeta^c\gamma^I$ is written with the terms of the canonical contact moment coordinates as $\zeta^c(\eta_i-\theta_i\zeta)^I$. It sufficient to see that to apply the operator $\partial_\zeta$ on the symbol $\zeta^c(\eta_i-\theta_i\zeta)^I$, this means that it is the same to apply the operator $\partial_\zeta-\theta_i\partial_{\gamma_i}$ on the symbol $\zeta^c\gamma^I$ and by writing the obtained result in the canonical moment coordinates. The second proof is similar.
\end{proof}
We define the following affine invariants.
\begin{df}
\begin{enumerate}
\item[(i)] We call contact divergence and we denote by $\mathrm{Div}_C$, the operator 
\[
\mathrm{Div}_C:\Sigma_\delta^{k,d}\to\Sigma_\delta^{k-1,d-1}:S\mapsto \partial_x\partial_\zeta S;
\]
\item[(ii)] We call the tangential divergence and we denote by $\mathrm{Div}_T$, the operator 
\[ 
\mathrm{Div}_T:\Sigma_\delta^{k,d}\to\Sigma_\delta^{k-1,d-\frac{1}{2}}:S\mapsto \overline{D}_r\partial_{\gamma_r}S.
\]
\end{enumerate}
\end{df}
The operators $\mathrm{Div}_C$ and $\mathrm{Div}_T$ intertwines the action of the affine Lie superalgebra $\mathrm{Aff}(2|n)$.\\

We have now the ingredients to compute the Casimir operator on the differential operators $\mathcal{D}_{\lambda\mu}^k(S^{1|n})$.
\begin{prop}
The Casimir operator $\mathcal{C}^\mathcal{D}$ on the differential operators $\mathcal{D}_{\lambda\mu}^k(S^{1|n})$ is written as
\[
\mathcal{C}^\D=C^S+N_{\mathcal{S}\mathcal{D}},
\] where the operator $N_{\mathcal{S}\mathcal{D}}$ equals to
\[ \frac{(2\lambda+k-1)}{2}\left(2\mathrm{Div}_C+\mathrm{Div}_T\right).\]
\end{prop}

\begin{proof}
By using the definition of $\gamma(X_f)$ (see \eqref{Affinemap} ) and the fact that it is equal to zero on the elements of $\mathrm{Aff}(2|n)$ and the definition of the Casimir operator (see Definition \eqref{Casimir}), we obtain the formula
\[
\mathcal{C}^\D=C^S-\frac{1}{2}\gamma(X_{x^2})L^S_{X_1}-\sum^n_{i=1}{\gamma(X_{x\theta_i})L^S_{X_{\theta_i}}}+\sum^n_{i=1}{L^S_{X_{\theta_i}}\gamma(X_{x\theta_i})}-\frac{1}{2}L^S_{X_1}\gamma(X_{x^2})\\
\]
By defining the operator $N_{\mathcal{S}\mathcal{D}}$ as follow
\[
N_{\mathcal{S}\mathcal{D}}=-\frac{1}{2}\gamma(X_{x^2})L^S_{X_1}-\sum^n_{i=1}{\gamma(X_{x\theta_i})L^S_{X_{\theta_i}}}+\sum^n_{i=1}{L^S_{X_{\theta_i}}\gamma(X_{x\theta_i})}-\frac{1}{2}L^S_{X_1}\gamma(X_{x^2})
\]
and by using the Lemma \ref{Gamma2} and the Proposition \ref{coucou}, we obtain the contributions of the all terms in the operator $N_{\mathcal{S}\mathcal{D}}$. The sum of these contributions gives the result.
\end{proof}
\begin{rmk}
It is easy to see that the operator $N_{\mathcal{P}\mathcal{D}}$ which measure the difference between the Casimir operator on the fine symbols $\mathcal{P}_\delta(S^{1|n})$ and the Casimir operator on the differential operators $\mathcal{D}^k_{\lambda\mu}(S^{1|n})$ is written as follow
\[
N_{\mathcal{P}\mathcal{D}}=\frac{1}{2}\Delta+\frac{2\lambda+k-1}{2}(2\mathrm{Div}_C+\mathrm{Div}_T).
\]
\end{rmk}
\subsection{Uniqueness of fine quantization}
We define a set of critical values of $\delta$.
\begin{df}\label{Crit-delta}
A value of $\delta$ is called critical if it exists the values $k,d,k'$ and $d'$ which satisfy the conditions
\[
d'\leqslant k'\leqslant 2d',\quad  d'< d \quad\mbox{and}\quad
d\leqslant k\leqslant 2d.
\] such that 
\[
\alpha_{k,d}-\alpha_{k',d'}= 0.
\]
\end{df}
We compute the explicit critical values in the following proposition.
\begin{prop}
If $P(k,k',,d,d')$ denotes the function
\[
(k-k')(2(k+k')+m-1)-4(kd-k'd')+2(d-d')(2d+2d'+1), 
\]then  the set $C_{\delta,\mathrm{crit}}$ of critical values of $\delta$ is given by
\[
\lbrace\frac{P(k,k',d,d')}{4(d-d')},  k,k'\in\N, d,d'\in\frac{1}{2}\N,  
 d'\leqslant k'\leqslant 2d',  d'< d,
d\leqslant k\leqslant 2d\rbrace.
\]
\end{prop}
\begin{proof}
It sufficient to use the Definition \ref{Crit-delta} and the formula \eqref{FUNCTION_VALUE} to conclude. 
\end{proof}
The main result of the uniqueness is given by the following theorem.
\begin{theorem}
If  $\delta\notin C_{\delta,\mathrm{crit}} $ and if $\delta\notin I_\delta\cup\mathsf{C}$, then the fine $\spo(2|n)$-equivariant quantization on $\mathcal{P}_\delta(S^{1|n})$ given in the Theorem \ref{Existence} is unique.
\end{theorem}
\begin{proof}
If $S\in\Sigma^{k,d}_\delta(S^{1|n})$, then it exists an unique $\hat{S}=S+\sum_{\substack{k',d'\\d'\leqslant k'\leqslant 2d'\\d'<d}}S_{k',d'}$ which the Heisenberg symbol is given by $S$ and which is an eigenvector with eigenvalue $\alpha_{k,d}$ of $\mathcal{C}^\mathcal{D}$. These conditions are written as 
\begin{multline}\label{EQUATION4}
(C^\Sigma+N_{\mathcal{P}\D})\left(S+\sum_{\substack{k',d'\\d'\leqslant k'\leqslant 2d'\\d'<d}}S_{k',d'}\right)
=\alpha_{k,d}\left(S+\sum_{\substack{k',d'\\d'\leqslant k'\leqslant 2d'\\d'<d}}S_{k',d'}\right).\\
\end{multline}
Since the operator $N_{\mathcal{P}\mathcal{D}}$ maps the space $\Sigma_\delta^{k,d}(S^{1|n})$ to the direct sum 
\[
\Sigma_\delta^{k,d-\frac{1}{2}}(S^{1|n})\oplus\Sigma_\delta^{k-1,d-\frac{1}{2}}(S^{1|n})\oplus\Sigma_\delta^{k-1,d-1}(S^{1|n}),
\] 

we obtain thus, in general, the equations in function of symbols
 
\begin{equation}\label{UNIC}
(\alpha_{k,d}-\alpha_{k',d'})S_{k',d'}= f(S_{k'',d''}),\quad\mbox{avec}\quad k''+d''>k'+d'
\end{equation}
and $f$ a function which depends on the operator $N_{\mathcal{P}\mathcal{D}}$.\\

Since $\delta$ is not critical, the expressions $(\alpha_{k,d}-\alpha_{k',d'})$
in \eqref{UNIC} are not equal to zero and every equation of the equations \eqref{UNIC} possesses an unique solution i.e, the terms $S_{k',d'}$ are determined with an unique way. We determine the terms $S_{k',d'}$ one by one.

Briefly, if $S\in \Sigma^{k,d}_\delta$, it exists an unique $\hat{S}$ such that $h\sigma^d(\hat{S})=S$ and such that $\mathcal{C}^\mathcal{D}(\hat{S})=\alpha_{k,d}(\hat{S})$. We define thus $Q(S)$ as $\hat{S}$.\\
If now $S\in\mathcal{P}^d_\delta(S^{1|n})$ and if  $S=\sum_{k=\lceil d\rceil}^{2d}S_k$ with $S_k\in\Sigma^{k,d}_\delta$, then one sets   $Q(S):=\sum_{k=\lceil d\rceil}^{2d}Q(S_k)$.
Since we have 
 \[
h\sigma^d(\mathcal{L}_{X_f}^{\lambda\mu}Q(S_k))= L_{X_f}^{\mathcal{P}}h\sigma^dQ(S_k)=L_{X_f}^{\mathcal{P}}(S_k),\quad\forall k \quad\mbox{and}\quad
h\sigma^d(Q(L_{X_f}^{\mathcal{P}}(S_k)))=L_{X_f}^{\mathcal{P}}(S_k),\forall k
\]  
then we obtain
\[
\mathcal{L}_{X_f}^{\lambda\mu}Q(S_k)=QL_{X_f}^{\mathcal{P}}(S_k),\quad\forall k
\] because the two members are the eigenvectors of $C^{\mathcal{D}}$ of the same eigenvalue. This means that one has:
\begin{multline*}
\quad \quad \mathcal{C}^{\D}(\mathcal{L}_{X_f}^{\lambda\mu}Q(S_k))=\mathcal{L}_{X_f}^{\lambda\mu}\mathcal{C}^{\D}Q(S_k) =\alpha_{k,d}\mathcal{L}_{X_f}^{\lambda\mu}Q(S_k)),\quad\forall k\\ \mbox{whereas}\quad
\mathcal{C}^{\D}Q(L_{X_f}^{\mathcal{P}}(S_k))=\alpha_{k,d}Q(L_{X_f}^{\mathcal{P}}(S_k)),\quad\forall k.\quad\quad\quad\quad
\end{multline*}
We have that   $L_{X_f}^{\mathcal{P}}(S_k)$ is the eigenvector of $C^\Sigma$ with eigenvalue $\alpha_{k,d}$ because
 \[ C^\Sigma(L_{X_f}^{\mathcal{P}}(S_k))=L_{X_f}^{\mathcal{P}}(\alpha_{k,d}S_k)=\alpha_{k,d}L_{X_f}^{\mathcal{P}}(S_k).
 \]
 
 By linearity, one has  $\mathcal{L}_{X_f}^{\lambda\mu}Q(S)=QL_{X_f}^\mathcal{P}(S)$.\\
If $Q$ is a fine quantization on $\mathcal{P}_\delta^d(S^{1|n})$, then its restrictions on the sub-spaces $\Sigma^{k,d}_\delta(S^{1|n})$ where
 $\lceil d\rceil\leqslant k\leqslant 2d$, are unique because if $S\in\Sigma^{k,d}_\delta(S^{1|n})$ then $Q(S)$ must be such that   $h\sigma^d(Q(S))=S$
 and  $\mathcal{L}_{X_f}^{\lambda\mu}Q(S)=Q(L_{X_f}^\mathcal{P}S)$, therefore $\mathcal{C}^\mathcal{D}(Q(S))$ must equals to $\alpha_{k,d}Q(S)$.
 
 The image $Q(S)$ is unique by using the first part of the proof. Since the restriction of $Q$ on the sub-spaces $\Sigma^{k,d}_\delta(S^{1|n})$ is unique then $Q$ is unique on $\mathcal{P}^d_\delta(S^{1|n})$.\\
 If $\delta\notin C_{\delta,\mathrm{crit}}$, then the existence and the uniqueness of the quantization are assured on the space  $\mathcal{P}^d_\delta(S^{1|n})$, for all $d\in\frac{1}{2}\N$.
Therefore, the equivariant quantization exists and it is unique on the space
 $\mathcal{P}_\delta(S^{1|n})$.
\end{proof}
\section{Acknowledgments}
It is a pleasure to thank F. Radoux, P. Mathonet, and J.P. Michel for helpful suggestions and discussions and V.Ovsienko for his interest to this work.


\begin{thebibliography}{10}

\bibitem{Nib14}
A.~Nibirantiza.
\newblock {\em Sur les quantifications équivariantes en supergéométrie de contact}.
\newblock {Th\`ese de Doctorat, Universit\'e de Li\`ege.}
 http://hdl.handle.net/2268/169770, 2014.

\bibitem{Nib15}
A.~Nibirantiza.
\newblock { On the matrix realization of Lie superalgebra of contact projective vector fields $\spo(2l+2|n)$}.
\newblock { Fundamental journal of mathematical physics,} Vol.4, issues 1-2, 2015, pages 53-70.
\newblock {http://www.frdint.com/} 
 
\bibitem{CoOv12}
C.H.~Conley and V.~Ovsienko.
\newblock Linear Differential Operators on Contact manifolds.
\newblock {\em arxiv:math-Ph/1205.6562v1},24p, 2012.

\bibitem{DuvLecOvs99}
    ~C. Duval, and P. Lecomte, and V. Ovsienko.
     \newblock{Conformally equivariant quantization: existence and
              uniqueness},
   \newblock{Ann. Inst. Fourier (Grenoble)},
  \newblock{Universit\'e de Grenoble. Annales de l'Institut Fourier},49(6),1999--2029, 1999.

\bibitem{Lei80}
D.~A.~Leites.
\newblock { Introduction to the theory of Supermanifolds}. 
\newblock Russian Math.Surveys 35:1 (1980),1-64.

\bibitem{Ber87}
F.~A.~Berezin.
\newblock { Introduction to superanalysis}, volume~9 of {\em Mathematical
  Physics and Applied Mathematics}.
\newblock D. Reidel Publishing Co., Dordrecht, 1987.
\newblock Edited and with a foreword by A. A. Kirillov, With an appendix by V.
  I. Ogievetsky, Translated from the Russian by J. Niederle and R. Koteck{\'y},
  Translation edited by D.~Leites.

\bibitem{GarMelOvs07}
H.~Gargoubi, N.~Mellouli, and V.~Ovsienko.
\newblock {Differential operators on supercircle: conformally equivariant
  quantization and symbol calculus.}
\newblock {\em Lett. Math. Phys.}, 79(1):51--65, 2007.

\bibitem{Gr13}
    ~J.Grabowski.
     \newblock{Graded contact manifolds and contact courant algebroids.}
   \newblock{Journal of Geometry and Physics},68(2013),27-58.
\bibitem{Mic09}
J.-P. Michel,
Quantification conform\'ement \'equivariante des fibr\'es supercotangents,
\href{http://tel.archives-ouvertes.fr/tel-00425576_v1/}{\textit{Th\`ese de Doctorat, Universit\'e Aix-Marseille II}} (2009).

\bibitem{MelNibRad13}
    N.~Mellouli, and A. Nibirantiza,  and F. Radoux.
    \newblock{$\spo(2|2)$-Equivariant Quantizations on the Supercircle $S^{1|2}$}.   \newblock{SIGMA 9 (2013), 055, 17 pages}
     \newblock{ http://dx.doi.org/10.3842/SIGMA.2013.055 }, arxiv:math.DG/1302.3727v2
     
\bibitem{Mel09}
N.~Mellouli.
\newblock Second-order conformally equivariant quantization in dimension $1|2$.
\newblock {\em SIGMA}, 5(111), 2009.

\bibitem{MatRad11}
P.~Mathonet and F.~Radoux.
\newblock Projectively equivariant quantizations over the Superspace $\R^{p|q}$.
\newblock {\em Lett. Math. Phys.}, 98:311--331, 2011.

\bibitem{Lecras}
P. B.~A. Lecomte.
\newblock Classification projective des espaces d'op\'erateurs diff\'erentiels
  agissant sur les densit\'es.
\newblock {\em C. R. Acad. Sci. Paris S\'er. I Math.}, 328(4):287--290, 1999.

\bibitem{Leconj}
P. B.~A. Lecomte.
\newblock Towards projectively equivariant quantization.
\newblock {\em Progr. Theoret. Phys. Suppl.}, (144):125--132, 2001.
\newblock Noncommutative geometry and string theory (Yokohama, 2001).
\bibitem{LO}
P.~B.~A. Lecomte and V.~Yu. Ovsienko.
\newblock Projectively equivariant symbol calculus.
\newblock {\em Lett. Math. Phys.}, 49(3):173--196, 1999.

\bibitem{MR2}
P.~Mathonet and F.~Radoux.
\newblock On natural and conformally equivariant quantizations.
\newblock {\em J. Lond. Math. Soc. (2)}, 80(1):256--272, 2009.

\bibitem{sarah}
S. Hansoul.
\newblock Existence of natural and projectively equivariant quantizations.
\newblock {\em Adv. Math.}, 214(2):832--864, 2007.
\bibitem{LMR}
T.~Leuther, P.~Mathonet and F.~Radoux.
\newblock On $\mathfrak{osp}(p+1,q+1|2r)$-equivariant quantizations.
\newblock {\em J. Geom. Phys.}, 62:87--99, 2012.
\bibitem{LR}
T.~Leuther and F.~Radoux.
\newblock Natural and Projectively Invariant Quantizations on Supermanifolds.
\newblock {\em SIGMA}, 7(34):12 pages, 2011.

 \bibitem{KacAdvances}
V.~G. Kac.
\newblock Lie superalgebras.
\newblock {\em Advances in Math.}, 26(1):8--96, 1977.

\end{thebibliography}
\end{document}